\documentclass[11 pt]{article}

\usepackage{latexsym,geometry}
\usepackage{amsmath,amsfonts,amssymb,amsthm}
\usepackage{graphicx,color,epsfig,curves}

\def\R{\mathbb{R}}
\def\N{\mathbb{N}}

\def\leq{\leqslant}
\def\geq{\geqslant}
\def\vf{\varphi}

\newcommand{\ve}{\varepsilon}

\newcommand{\Lp}{\mathcal{L}}
\newcommand{\B}{\mathcal{B}}

\newcommand{\cJ}{\mathcal{J}}
\newcommand{\F}{\mathcal{F}}

\newcommand{\hI}{\mathring{I}}
\newcommand{\hT}{\mathring{T}}

\newcommand{\hLp}{\mathring{\Lp}}
\newcommand{\hX}{\mathring{X}}

\newcommand{\musrb}{\mu_{\mbox{\tiny SRB}}}
\newcommand{\e}{\mathfrak{e}}
\newcommand{\epoly}{\e_{\mbox{\scriptsize poly}}}

\newtheorem{Pro}{Proposition}[section]
\newtheorem{Lem}[Pro]{Lemma}

\newtheorem{Thm}[Pro]{Theorem}

\newtheorem{Rem}[Pro]{Remark}

\begin{document}

\title{Escape Rates and Singular Limiting Distributions for Intermittent Maps with Holes}
\author{Mark Demers$^1$ and Bastien Fernandez$^2$}
\date{}   
\maketitle
\begin{center}
$^1$ Department of Mathematics  \\
Fairfield University \\
Fairfield, CT 06824 USA

$^2$ Centre de Physique Th\'{e}orique\\ CNRS - Aix-Marseille Universit\'e - Universit\'e de Toulon\\ Campus de Luminy\\ 13288 Marseille CEDEX 9, France
\end{center}

\begin{abstract}
We study the escape dynamics in the presence of a hole of a standard family of intermittent maps of the unit interval with neutral fixed point at the origin (and finite absolutely continuous invariant measure). Provided that the hole (is a cylinder that) does not contain any neighborhood of the origin, the surviving volume is shown to decay at polynomial speed with time. The associated polynomial escape rate depends on the density of the initial distribution, more precisely, on its behavior in the vicinity of the origin. Moreover, the associated normalized push forward measures are proved to converge to the point mass supported at the origin, in sharp contrast to systems with exponential escape rate. Finally, a similar result is obtained for more general systems  with subexponential escape rates; namely that the Ces\`aro limit of normalized push forward measures is typically singular, invariant and supported on the asymptotic survivor set.
\end{abstract}

\leftline{\small\today.}
\bigskip


\section{Introduction and setting}
The study of systems with holes finds its origin in the study of Markov
chains with absorbing states \cite{ferrari, seneta, vere, yaglom} and was introduced
in deterministic dynamical systems by Pianigiani and Yorke \cite{pianigiani yorke}.
It has focused on the establishment of escape
rates and on the existence of conditionally invariant measures which describe the
asymptotic distribution of mass conditioned on non-escape. 

Since conditionally invariant measures are badly non-unique \cite{demers young},
physically relevant measures are usually characterized as the limit of normalized push forward iterates of a reference measure (usually Lebesgue). 
Such limiting distributions are typically eigenmeasures with maximal
eigenvalue of the corresponding transfer operator defined on an appropriate function
space. The maximal eigenvalue itself gives the exponential rate of escape of mass from the system.
These limiting conditionally invariant measures have properties analogous to
Sinai-Ruelle-Bowen (SRB) measures for the corresponding closed  system. Under reasonable assumptions, they converge to the SRB measure as the size of the
hole tends to zero, and this establishes stability under perturbations in the form of holes.
 
Examples 
begin with open systems admitting a finite Markov partition:  expanding maps in $\R^n$ \cite{collet ms1, pianigiani yorke},
Smale horseshoes \cite{cencova}, Anosov diffeomorphisms 
\cite{chernov m1, chernov m2},
and some unimodal maps \cite{h young}. Subsequent attempts to substitute the Markov
assumption by requiring that the holes be small have extended this analysis
to Anosov diffeomorphisms with non-Markov holes \cite{chernov mt1, chernov mt2},  to expanding maps of the interval \cite{chernov bedem, demers exp, liverani maume}, to multimodal maps satisfying a Collet-Eckmann condition \cite{bdm, demers logistic}, to piecewise
hyperbolic maps \cite{demers liverani}
and recently, to various classes of dispersing billiards \cite{dwy1, demers billiard, demers infinite}.

The characteristic common to all these systems is that the rate of escape is exponential
(the systems enjoy exponential decay of correlations before the introduction of the hole)
so that the concept of conditionally invariant measure is well-defined.  

Polynomial rates of escape have been studied numerically in some systems 
\cite{dettmann, dettmann leonel} and via formal expansions to obtain leading order terms
for the decay rate \cite{buni dett 1, buni dett 2}. However, to our knowledge, there are no analytical results regarding
limiting distributions for systems with polynomial rates of escape. 

The purpose of the present paper is to initiate the rigorous mathematical analysis of open systems with subexponential rates of escape. For simplicity, we consider a family of intermittent maps $T$ of the unit interval, with neutral fixed point at the origin \cite{LSV99}. For the hole, we take any element of a refined Markov partition for the map, not adjacent to the origin. (Of note, \cite{froyland} has also considered interval maps with neutral fixed point and very specific holes which are either a neighborhood of the neutral fixed point or its complement.) 

In this context, we first prove that the rates of escape must be polynomial for a large class of initial distributions, and this rate depends on the behavior of the initial distribution in a neighborhood of the origin. In particular, the polynomial rate of escape with respect to the SRB measure 
(for the map before the introduction of the hole) differs from that with respect to Lebesgue measure.  

In this setting, conditionally invariant measures
are not physically meaningful (although plenty still exist with any desired 
eigenvalue between 0 and 1 \cite{demers young}). 
Letting 
$\hT$ denote the map with the hole, we show that the limit of $\hT^n_*\mu/|\hT^n_*\mu|$ (NB: for the precise definition of 
this notation, see section \ref{setting} below) converges to the point mass at the neutral fixed point 
for a large class of initial distributions $\mu$ (including both Lebesgue and the SRB measures).  

These results hold independently of the size of the hole. Thus from the point of view of the physical limit $\hT^n_*\mu/|\hT^n_*\mu|$, a hole of any size is always a large perturbation in the context of subexponentially mixing systems.  In other words, the attracting property of the SRB measure under the action of $T^n_*$ is unstable with respect to small leaks in the system.

Finally, we consider more general systems with subexponential rates of
escape. The analysis of intermittent maps of the interval might suggest that the results 
are specific to this setting. Our final result Theorem~\ref{thm:mean conv} shows that this is not the 
case: in contrast to situations with exponential escape, in systems with slow escape, the (Ces\`aro) 
limit of $\hT^n_* \mu / |\hT^n_* \mu|$ for reasonable reference measures $\mu$ 
will always be singular and will typically be supported on the survivor set
of points that never escape.

\subsection{Setting}
\label{setting}
We study the dynamics of the family of maps of the unit interval $T:I\to I$ where $I=[0,1)$ and $T$ is defined by (see \cite{LSV99} and Figure \ref{LSVMap})
\[
T(x)=\left\{\begin{array}{ccl}
x + 2^\gamma x^{\gamma+1}&\text{if}&x \in [0, \tfrac12)\\
2x-1&\text{if}&x \in [\tfrac12, 1)
\end{array}\right.
\]
with $0 < \gamma < 1$, after the introduction of a hole $H$ into $I$. 
\begin{figure}[ht]
\centerline{\includegraphics*[height=65mm]{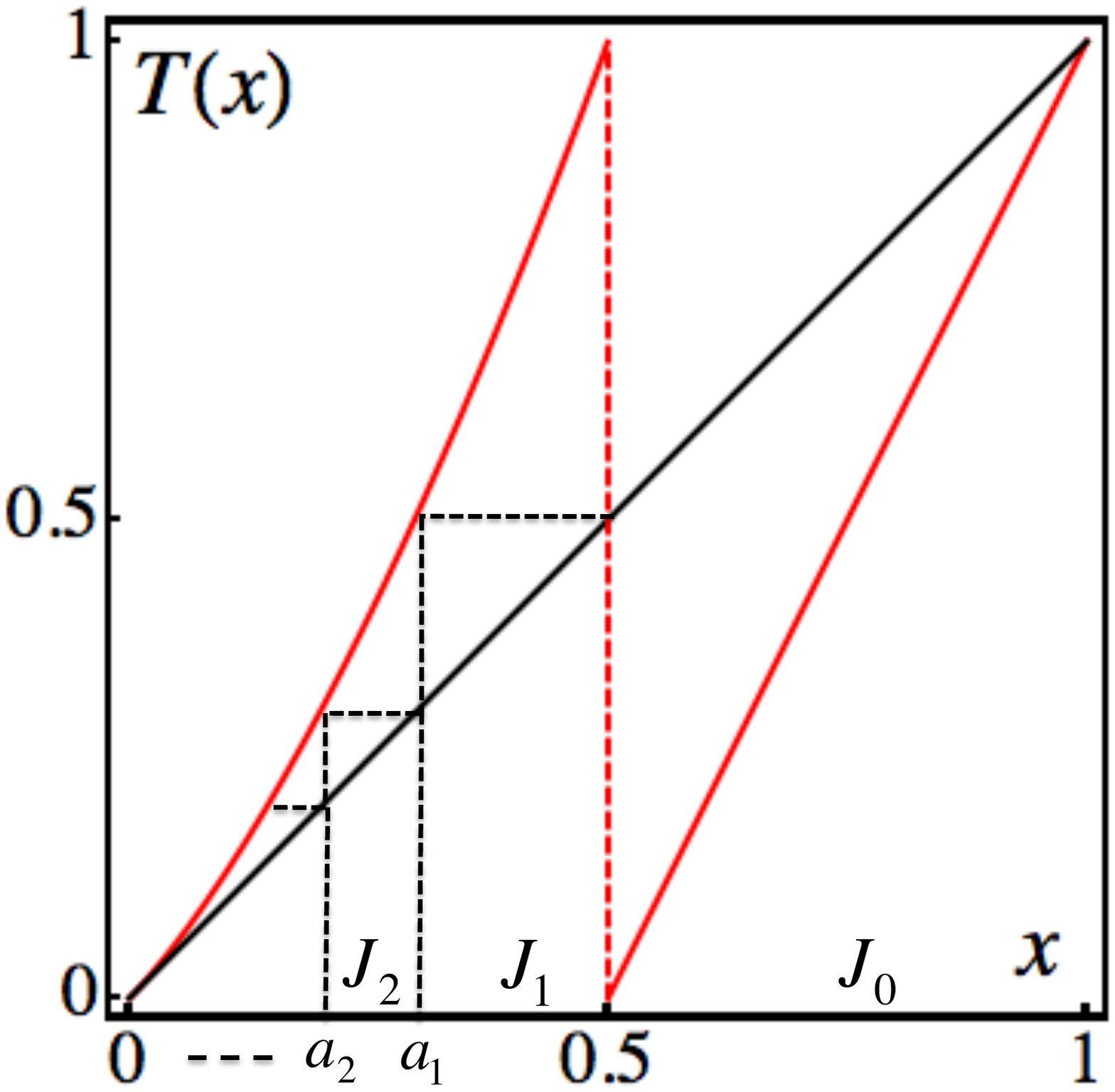}}
\caption{Graph of the map $T$ for $\gamma=\tfrac34$ (solid red branches), together with some the intervals $J_n=[a_n,a_{n-1})$.}
\label{LSVMap}
\end{figure}
 In this parameter range, $T$ preserves a finite invariant
measure, $\musrb$, absolutely
continuous with respect to Lebesgue.  

In order to define the hole, we need to introduce the (standard) finite and countable Markov partitions of $I$. The finite partition is defined by $\mathcal{P}:=\{ J_L, J_R \}$ where $J_L = [0,\tfrac12)$ and $J_R = [\tfrac12,1)$. The countable partition is defined by $\cJ:=\{ J_n \}_{n \geq 0}$ where 
\[
J_n=\left\{\begin{array}{ccl}
J_R&\text{if}&n=0\\
\left[a_n,a_{n-1}\right)&\text{if}&n \geq 1
\end{array}\right.
\]
where $a_n = T_L^{-n}(\tfrac12)$ and $T_L$ denotes the left branch of $T$. (Note that $T(J_n)=J_{n-1}$ for all $n\geq 1$, see Figure \ref{LSVMap}.)

Now, given $t\geq 0$, let $\cJ^{(t)}$ be the refined partition defined as follows 
\[
\cJ^{(t)}:=\cJ\vee\bigvee_{i=0}^tT^{-i}(\mathcal{P}).
\]
The hole $H$ is defined to be any element of $\cJ^{(\ell_H)}$ where $\ell_H\geq 0$ is arbitrary. 
We shall denote by $J_h\supseteq H$ with $h\geq0$, the element of ${\mathcal J}$ that contains $H$.

This assumption on the hole gives immediate access to a countable Markov partition for the 
open system. 
Notice that dynamically refining $\cJ$ according to $T^{-i}(\mathcal{P})$ - and not $T^{-i}(\cJ)$ - preserves $0$ as the only accumulation point of the endpoints
of elements of $\cJ^{(t)}$ for all $t \geq 0$.
In particular, this property is convenient for the conditioning arguments in the proof of 
Lemma~\ref{lem:upper bound} and for the invariance of a certain function space used 
in the proof of Theorem~\ref{thm:convergence} that implies control of the
structure of the singular limit. 
Nevertheless, we believe the assumption that 
$H$ be an element of $\cJ^{(\ell_H)}$ is purely technical and we expect our results 
to hold even when relaxed, although significant technical modifications will have to be made.

Define $\hI = I \setminus H$
and given $t\geq 0$, let $\hI^t = \bigcap\limits_{i=0}^t T^{-i}(I\setminus H)$ represent those points which have not
escaped by time $t$ (NB: we have $\hI^0=\hI$).  We refer to $\hT := T|_{\hI^1}$ as the map with a hole
and its iterates $\hT^t = T^t|_{\hI^{t}}$ ($t\geq 1$) describe the dynamics
of the open system before escape. Notice that $\cJ^{(\ell_H)}$ is also a countable Markov partition for $\hT$.

One of the quantities we will be interested in studying is the rate of escape
of mass from the open system.  Given a measure $\mu$ on $\hI$, we
define the \emph{polynomial rate of escape} with respect to $\mu$ by
\[
\epoly(\mu) = - \lim_{t \to +\infty} \frac{\log \mu(\hI^t)}{\log t},
\]
whenever the limit exists. 

We will also study the asymptotic evolution of absolutely continuous measures that are transported under the action of $\hT$. Given a measure $\mu$ on $\hI$ and $t\geq 1$, let $\hT^t_\ast\mu$ be the push forward measure under the action of $\hT^t$. 

Let $m$ denote Lebesgue measure on $I$ and given  $f\in L^1(m)$, let $\mu_f$ be the 
absolutely continuous measure  with density $f$. Let $\Lp$ be the transfer operator associated with $T$ defined by the expression
\[
\Lp f(x) = \sum_{y \in T^{-1}(x)} \frac{f(y)}{DT(y)},
\]
where $DT>0$ is the (first) derivative of $T$.
Consider the operators 
$\hLp^tf := \Lp^t(f 1_{\hI^t})$ where $\Lp^t$ are the iterates of $\Lp$
and $1_A$ denotes the characteristic function of the set $A$. 
We have $\hT^t_\ast\mu_f=\mu_{\hLp^tf}$ for all $t\geq 0$ and the change of 
variable formula implies in this case the following relation
\[
|\hLp^t f |_1=\mu_f(\hI^t),\ \forall t\geq 0,
\]
where $| \cdot |_1$ denotes the $L^1$-norm with respect to Lebesgue measure $m$. 
It will be useful for us later that, with these definitions, the usual composition property of the
transfer operators $\hLp^t$ holds, i.e. for any $j, k \geq 1$,
\[
\hLp^j(\hLp^k f)(x) = \sum_{y \in T^{-j}(x)} \frac{(\hLp^kf)(y)}{|DT^j(y)|} 1_{\hI^j}(y)
= \sum_{z \in T^{-(j+k)}(x)} \frac{f(z)}{|DT^{j+k}(z)|} 1_{\hI^j}(T^k(z)) 1_{\hI^k}(z) = \hLp^{j+k}f(x),
\]
where the last equality follows from the fact that $1_{\hI^j}(T^k(z)) 1_{\hI^k}(z)=1_{\hI^{j+k}}(z) $. 


\section{Statement of Results}
\label{results}


\subsection{Results for the open system $\hT$}
Throughout this section, the hole $H$ is fixed as in the previous section (and so are $\ell_H$ and $h$). Our first result describes a common set of escape rates for initial distributions depending on
their behavior near 0. Following \cite{young polynomial}, the notation $u_t \approx v_t$ (resp.\ $u_t\lesssim v_t$, $u_t\gtrsim v_t$) means there exists $C>0$ such that $C^{-1}v_t\leq u_t\leq Cv_t$ (resp.\ $u_t\leq C v_t$, $u_t\geq C v_t$) for all $t$. These notations will also be employed as abbreviations for uniform estimates on sequences with multiple indices. 
\begin{Thm}
\label{thm:escape}
For any non-negative $f\in L^1(m)$ for which there exist $x_0\in (0,1)$ and $\alpha\in [0,1)$ such that
\begin{equation}
\label{eq:order}
0<\inf_{x\in (0,x_0)}x^\alpha f(x)\leq \sup_{x\in (0,1)}x^\alpha f(x)<+\infty,
\end{equation}
we have 
\[
\mu_f(\hI^t) \approx t^{-\frac{1-\alpha}{\gamma}}.
\]
Consequently, the associated measure $\mu_f$ has polynomial escape rate, $\epoly(\mu_f) = \frac{1-\alpha}{\gamma}$.
\end{Thm}
The proof is given in Section \ref{escape}. Of note, to obtain the lower bound on $\mu_f(\hI^t)$ is rather immediate (see relation \eqref{eq:lower}). Moreover, ergodicity of the map $T$ 
with respect to the absolutely continuous invariant measure $\musrb$
implies $\mu_f(\hI^t)\xrightarrow{t\to +\infty} 0$. Thus, most of the proof consists in proving the upper bound. This part is inspired by the proof in \cite{young polynomial} of the speed of convergence to the equilibrium measure.
 
Theorem \ref{thm:escape} implies in particular that the polynomial escape rate associated with  Lebesgue measure is given by $\epoly(m) = \frac 1\gamma$. Interestingly, since $\frac{d\musrb}{dm}(x) \approx x^{-\gamma}$ for $x$ near 0 \cite{young polynomial}, this rate differs from the one associated with the SRB measure, $\epoly(\musrb) = \frac{1-\gamma}{\gamma}$.

That the escape rate is polynomial depends on the assumptions both on $H$ and on the initial density $f$. Indeed, if the hole included a neighborhood of the neutral fixed point 0, then the corresponding open system $\hT$ would be uniformly expanding and the escape rate would be exponential for any initial density $f\in L^1(m)$; see \cite{froyland,pianigiani yorke} for the Markov case and any of \cite{chernov bedem,demers exp,liverani maume} for the non-Markov case. (Obviously, such holes do not belong to $\cJ^{(t)}$ for any $t\geq 0$.) 

Alternatively, $H = [\tfrac12, \tfrac12+d_{\ell_H})\in \cJ^{(\ell_H)}$ for some $d_{\ell_H}>0$ can also create exponential behavior for some initial densities $f$. Indeed, the map $T|_{[2d_{\ell_H}, 1)}$ is uniformly expanding 
and no point in $[2d_{\ell_H},1)$ can enter the interval $[0,2d_{\ell_H})$ without first falling into $H$. Hence, the measure associated with any smooth density $f$ satisfying $f|_{[0,2d_{\ell_H})} \equiv 0$ must experience an exponential rate of escape in this case. (Such densities do not satisfy the assumption of  Theorem~\ref{thm:escape}.)

However, for any hole not blocking repeated passes through a neighborhood of
0, $\hLp^t f$ will eventually be positive in a neighborhood of 0 (and bounded) for any `typical' smooth density $f$; hence Theorem \ref{thm:escape} implies that the associated measure will experience a polynomial escape rate $ \frac 1\gamma$.

Our next result describes the limiting behavior of the sequence 
$\left\{\frac{\hT^t_\ast\mu_f}{\mu_f(\hI^t)}\right\}_{t\in\N}$ of push forward probability 
measures, for initial densities $f$ that are log-H\"older continuous on elements of the partition $\cJ^{(\ell_H)}$. To be precise, let $C^0(\cJ^{(\ell_H)})$ denote the set of functions defined in the interior of $I$ and continuous on each element of $\cJ^{(\ell_H)}$.
Given $f \in C^0(\cJ^{(\ell_H)})$, $f \geq 0$, $p\in \R^+$ and $J \in \cJ^{(\ell_H)}$, 
define the quantity $H^p_J(f)$ as follows 
\[
H^p_J(f) = \left\{\begin{array}{ccl}
0&\text{if}&f\equiv 0\ \text{on}\ J\\
+\infty&\text{if}&f(x)=0<f(y)\ \text{for some}\ x,y\in J\\
{\displaystyle \sup_{x\neq y \in J} \frac{\log f(x) - \log f(y)}{|x-y|^p}}&\text{if}&f>0\ \text{on}\ J
\end{array}\right.
\]
and let  $\| f \|_p := \sup\limits_{J \in \cJ^{(\ell_H)}} H^p_J(f)$. Consider also the set of functions,
\[
\F_p = \{ f \in C^0(\cJ^{(\ell_H)})\ :\ f \geq 0,\ |f|_1=1\ \text{and}\ \| f \|_p < +\infty\}.
\]
and its subset
\[
\F_p^0 = \{ f \in \F_p : \exists x_0 \in (0,1) \mbox{ and } \alpha \in [0,1) \mbox{ such that } 
\eqref{eq:order} \mbox{ holds} \}.
\]
of functions which are bounded away from zero in a neighborhood of 0. 

\begin{Thm}
\label{thm:convergence}
Let $f \in \F_p^0$ for some $p > 0$.
Then the sequence $\left\{\frac{\hT^t_\ast\mu_f}{\mu_f(\hI^t)}\right\}_{t\in\N}$ of absolutely continuous measures behaves asymptotically as follows 
\[
\lim_{t \to +\infty} \frac{\hT^t_\ast\mu_f}{\mu_f(\hI^t)}= \delta_0 ,
\]
where $\delta_0$ denotes the point mass at 0 and the convergence is in the weak sense. Moreover, we have
\[
\lim\limits_{t \to +\infty} \frac{\mu_f(\hI^{t+1})}{\mu_f(\hI^t)}= 1.
\]
\end{Thm}
Of note, this last expression of the theorem can be alternatively formulated as 
\[
\lim\limits_{t \to +\infty} \frac{\mu_f(\hI^t\setminus\hI^{t+1})}{\mu_f(\hI^t)}= 0.
\]
Theorem \ref{thm:convergence} applies in particular to Lebesgue measure, since $1$ belongs to $\F_p^0$ for every $p\in\R^+$. (More generally, one easily checks that any H\"older 
continuous density with exponent $p$ which is bounded away from 0 on $I$ 
belongs to $\F_p^0$). Theorem \ref{thm:convergence} also applies to $\mu_{\mbox{\tiny SRB}}$ since the  density $f_{\mbox{\tiny SRB}}=\frac{d\mu_{\mbox{\tiny SRB}}}{dm}$ belongs to $\F_p^0$ for every $p\in (0,\frac{\gamma}{\gamma+1}]$ as does any normalized  density $f/|f|_1$, where $f(x)=x^{-\alpha}$ for some $\alpha\in (0,1)$ 
(see Lemma \ref{EXAMPLESSPACE} in Section \ref{space}). 

As mentioned in the introduction, this theorem implies that arbitrarily small holes in
systems with polynomial rates of escape can act as large perturbations
from the point of view of the physical limit $\frac{\hT^t_\ast\mu_f}{\mu_f(\hI^t)}$.

Furthermore, one may also consider the stability of open systems with respect to the location of a hole of a given size \cite{AB10,BY10, demers wright, KL09}. 
In this framework, consider a family of holes of the form $\{H_{\ve_i}\}$ where 
$\ve_i>0$, $\lim\limits_{i\to+\infty}\ve_i=0$ and 
$H_{\ve_i} = [\ve_i, \ve_i+\eta_i)$ all satisfy the assumptions above ($\eta_i>0$ is small). 
Then our results state that, for each $i$, the sequence $\frac{\hT^t_\ast m}{m(\hI^t)}$ 
tends to $\delta_0$ for large $t$.
However, for any $n\geq 1$ and $H_0 = [0,a_n)$, {\sl i.e.}\ $\ve_i=0$, the 
results of \cite{pianigiani yorke} imply that the escape rate is exponential 
and the sequence
$\frac{\hT^t_\ast m}{m(\hI^t)}$ tends to a conditionally invariant measure 
that is absolutely continuous with respect to Lebesgue. From this point of view, 
a discontinuity occurs when the hole goes through the neutral fixed point. 


\subsection{General open systems: Consequence of a subexponential escape rate}
\label{general}
The convergence of $\frac{\hT^t_\ast\mu_f}{\mu_f(\hI^t)}$ to a singular limit as in Theorem~\ref{thm:convergence} is not limited to the map $\hT$ above. Indeed, as we show now, this phenomenon occurs very generally when the rate of escape is subexponential.

To see this, let $X$ be a compact, separable metric space
 and let $T: X\to X$ be now an arbitrary Borel measurable map.
Assume there exists a Borel probability measure $\mu$ with respect to 
which $T$ is nonsingular (but not necessarily invariant).  This will be our reference measure.

Let an open set $H \subset X$ be the hole and 
let $\hX^t = \bigcap\limits_{i=0}^t T^{-i}(X \setminus H)$ denote the survivor set up until
time $t \in \mathbb{N} \cup \{ +\infty \}$. 
As before, let $\hT := T |_{\hX^1}$. We have $\hT^t = T^t |_{\hX^t}$ for all $t \geq 1$. Our main assumption on this open system 
is that $\mu$-almost every point escapes and that the escape rate
is subexponential, i.e. we assume 
\begin{equation}
\label{eq:slow}
\mu\left(\bigcup_{i=1}^\infty int\left(\bigcup_{j=0}^i T^{-j}(H)\right)\right)=1,\quad \mu(\hX^t)>0\ \text{for all}\ t\geq 0\quad \mbox{and}\ \limsup_{t \to +\infty} \frac{\log \mu(\hX^t)}{t} = 0 ,
\end{equation}
where $int(A)$ denotes the interior of a set $A$.
In particular, this includes both polynomial and stretched exponential rates of escape (and does not assume $\liminf_{t \to +\infty} \frac{\log \mu(\hX^t)}{t} = 0$).
We remark that if $T$ is continuous, the first assumption in \eqref{eq:slow}
is equivalent to $\mu(\hX^\infty)=0$.

By assumption, all push forward (probability) measures $\frac{\hT^t_\ast\mu}{\mu(\hX^t) }$ are nonsingular with respect to $\mu$. Hence, the same is true for 
\begin{equation}
\mu_t = \frac{1}{t} \sum_{k=0}^{t-1} \frac{\hT^k_\ast\mu}{\mu(\hX^k)},
\label{CESARO}
\end{equation}
for all $t\geq 1$. As the next result shows, any limit point however must be singular. Convergence here is also understood in the weak sense.
\begin{Thm}
\label{thm:mean conv}
Any limit point $\mu_\infty$ of the sequence $\{ \mu_t \}_{t \in \N}$ 
is singular with respect to $\mu$ and is supported on $X \setminus \left(\bigcup\limits_{i=1}^\infty int\left(\bigcup\limits_{j=0}^i T^{-j}(H)\right)\right)$.  If, in addition,
$\mu_\infty$ gives zero measure to the discontinuity set of $\hT$, then 
$\mu_\infty$ is $T$-invariant and supported on $\hX^\infty$.
\end{Thm}
Interestingly, the averaging method presented here does not
work so easily in the case of exponential escape (unless a priori one knows that the
limit of $\frac{\hT^t_\ast\mu}{\mu(\hX^t)}$ itself exists). Indeed, in this case,  
the ratio of consecutive normalizations $\frac{\mu(\hX^{t+1})}{\mu(\hX^t)}$
does not converge to 1 and the terms appearing in the sum must be weighed to
compensate for this.  For an example of an averaging method in the exponential
case under stronger assumptions, see \cite{collet existence, corrigendum}.  
Theorem~\ref{thm:mean conv} is proved in Section~\ref{cesaro}.


\section{Proofs}
\label{proofs 12}


\subsection{Preliminary estimates}
In proving the theorems, we shall repeatedly use the following bounds \cite{young polynomial}
\begin{equation}
\label{eq:spacing}
 a_n \approx n^{-\frac 1\gamma} \qquad \mbox{and} 
\qquad  |J_n| \approx n^{-\frac{\gamma+1}{\gamma}}.
\end{equation}
We shall also rely on the following lemma.
\begin{Lem}
\label{lem:distortion}
Given $n\geq 0$ and $t\geq 1$, let $x,y \in \hI^t$ lie in the same element of $\cJ^{(t+1)}$ such that $T^t(x), T^t(y) \in J_n$. Then we have  
\begin{enumerate}
  \item[(a)]  $\frac{1}{DT^t(x)} \lesssim \left(\frac{n}{n+t} \right)^\frac{\gamma+1}{\gamma}$;
  \item[(b)]  for any $p \in (0, \frac{\gamma}{\gamma+1}]$, one has $\displaystyle \left| \log \frac{DT^t(x)}{DT^t(y)} \right| \lesssim |T^t(x)-T^t(y)|^p$.
\end{enumerate}
\end{Lem}

\begin{proof}
(a) Given $i\in \{0,\cdots,t-1\}$, let $J_{n_i}$ denote the element of $\cJ$ containing $T^i(x)$ and $T^i(y)$. Let also $B_{n_i}=2^\gamma\gamma(\gamma+1)a_{n_i}^{\gamma-1}$ be the maximum value of $|D^2T|$ and $M_{n_i}^{(j)}$ be the minimum value of $|DT^j|$ on $J_{n_i}$, respectively. We have 
\begin{align}
\left| \log \frac{DT^t(x)}{DT^t(y)} \right|&
\leq  \sum_{i=0}^{t-1} |\log DT\circ T^i(x) - \log DT\circ T^i(y)|\nonumber\\
&\leq \sum_{i=0}^{t-1} \frac{B_{n_i}}{M_{n_i}^{(1)}}|T^i(x) - T^i(y)|\leq \sum_{i=0}^{t-1} B_{n_i}|T^i(x) - T^i(y)|  \, .  \label{eq:first dist}
\end{align}
Following \cite{young polynomial}, we write $|T^t(x) - T^t(y)|=|DT^{t-i}(z)| |T^i(x) - T^i(y)|$ for some $z\in J_{n_i}$ and use that the expansion $DT^{t-i}(z)$ decreases as $n_i$ increases to conclude that the last sum here is maximised for $n_i=n+t-i$, i.e. $T^i(x),T^i(y)\in J_{n+t-i}$. Using equation \eqref{eq:spacing} we obtain 
\[
\left| \log \frac{DT^t(x)}{DT^t(y)} \right|\lesssim  \sum_{i=0}^{t-1} (n+t-i)^\frac{1-\gamma}{\gamma}(n+t-i)^{-\frac{\gamma+1}{\gamma}}= \sum_{i=0}^{t-1} (n+t-i)^{-2}\leq n^{-1},
\]
where the last inequality follows from $\sum\limits_{i=n+1}^{\infty}i^{-2}\leq\int\limits_{n}^{+\infty}x^{-2}dx$.

Now, $T^t(x)$ has no preimage in $\bigcup\limits_{i > n+t} J_i$, so the weakest expansion occurs when
$x \in J_{n+t}$. The previous distortion estimate implies 
\[
DT^t(x) \gtrsim \frac{|J_n|}{|J_{n+t}|}. 
\]
Using equation \eqref{eq:spacing} again, statement (a) easily follows. 
\medskip

\noindent
(b)  
Adopting the same notation as in (a) and starting from \eqref{eq:first dist}, we 
fix $p \in (0, \frac{\gamma}{\gamma+1}]$ and write
\[
\begin{split}
\left| \log \frac{DT^t(x)}{DT^t(y)} \right|
& \leq \sum_{i=0}^{t-1} B_{n_i}|T^i(x)-T^i(y)|^{1-p} 
\frac{|T^i(x) - T^i(y)|^p}{|T^t(x) - T^t(y)|^p} |T^t(x) - T^t(y)|^p \\
& \leq \sum_{i=0}^{t-1} \frac{B_{n_i} |J_{n_i}|^{1-p} }{(M^{(t-i)}_{n_i})^p} |T^t(x) - T^t(y)|^p .
\end{split}
\]
Using statement (a) and, as in the previous proof, that the worst case scenario in the upper bounds of equation \eqref{eq:spacing} occurs for $n_i=n+t-i$, we obtain 
\[
\left| \log \frac{DT^t(x)}{DT^t(y)} \right| \lesssim n^{p\frac{\gamma+1}{\gamma}}\sum_{i=0}^{t-1} (n+t-i)^{-2} |T^t(x)-T^t(y)|^p \leq n^{p\frac{\gamma+1}{\gamma}-1}|T^t(x)-T^t(y)|^p
\]
and statement (b) follows from the assumption $p\leq \frac{\gamma}{\gamma+1}$. 
\end{proof}

Finally, on several occasions in the proofs, we shall require the following estimate.

\begin{Lem}
\label{lem:double}
\[
\sum_{n=n_0+1}^{t+n_1}n^{-a}(t-n+n_0)^{-b}\lesssim t^{-\min\{a,b\}}
\]
for every pair $a,b> 1$, every pair $n_0,n_1$ such that $n_0>n_1+1$, and for all $t$ such that $n_0+1\leq t+n_1$.
\end{Lem}

\begin{proof}
 According to the inequality $\sum\limits_{n=n_0+1}^{t+n_1}f(n)\leq\int\limits_{n_0}^{t+n_1+1}f(x)dx$ which holds for every $f\geq 0$, we estimate the sum via the following integral,
\[
\int_{n_0}^{t+n_1+1} x^{-a}(t + n_0 - x)^{-b} \, dx
= \int_{n_0}^{t/2} x^{-a}(t + n_0 - x)^{-b} \, dx
+ \int_{t/2}^{t+n_1+1} x^{-a}(t + n_0 - x)^{-b} \, dx . 
\]
In the first integral, the second factor in the integrand is at most $(n_0 + \frac t2)^{-b}$ while
the first factor integrates to something less than $\frac{n_0^{-(a-1)}}{a-1}$.
In the second integral, the first factor is at most $(\frac t2)^{-a}$ while the
second factor integrates to something less than 
$\frac{(n_0-n_1-1)^{-(b-1)}}{b-1}$. The desired estimate immediately follows.
\end{proof}


\subsection{Estimating escape rates - proof of Theorem~\ref{thm:escape}}
\label{escape}
Recall that $H$ is a cylinder in $\cJ^{(\ell_H)}$ and $H \subseteq J_h$ for some $h\geq 0$. The main estimate of this section is the following lemma.
\begin{Lem}
\label{lem:upper bound}
$m(\hI^t) \lesssim t^{-\frac 1\gamma}$.
\end{Lem}
The proof of this lemma is based on the fact that an induced map
related to $T$ has exponential escape rate. To formulate this property, choose $n_S> h$, let $I_S=[a_{n_S},1)\supset H$ and consider the induced map $S = T^R :I_S\to I_S$, where $R$ is the first return time to $I_S$. 

Let $\hI_S^t=\bigcap\limits_{i=0}^tS^{-i}(I_S\setminus H)$ denote the set of points in $I_S$ which do not enter $H$ before
time $t$ under the action of $S$. The induced open system $S|_{\hI_S^0}$ is uniformly expanding with countably many branches and admits a countable Markov partition which is formed 
by joining $\cJ^{(\ell_H)}$ with the partition into sets on which $R$ is constant.  
The action of $S$ on this partition satisfies the large images condition \cite{demers exp}; hence the following property holds.\footnote{The full results of \cite{demers exp} also 
require a ``smallness" condition on the size of the hole.  This condition is not needed here
since we are not invoking any results regarding a spectral gap for the transfer operator
associated with $S$, but just an exponential rate of escape, which does not require the
hole to be small.}
\begin{Lem}
\label{lem:exp escape}
{\rm \cite{demers exp}} There exists $\sigma < 1$ such that $m(\hI_S^t) \lesssim \sigma^t$.
\end{Lem}

\begin{proof}[Proof of Lemma~\ref{lem:upper bound}]
We first assume there exists $d>0$ such that $[\frac 12, \frac 12 + d) \cap H = \emptyset$.
The complementary case is much simpler and will be addressed at the end of the proof.

Without loss of generality, we can choose the index $n_S$ that defines $I_S$ sufficiently large so that  $\frac{a_{n_S}}{2} < d$, {\sl viz.}\ the open system makes full returns to the interval $J_S=[\frac 12 , \frac 12 + \frac{a_{n_S}}{2} )$ before entering $I\setminus I_S=[0, a_{n_S})$.  
In order to obtain the estimate on $m(\hI^t)$, we consider separately the sets $\hI^t\cap I_S$ and $\hI^t\cap (I\setminus I_S)$. 

\noindent
{\em Case I. Estimate for points in $\hI^t\cap I_S$.}
Consider the decomposition of $\hI^t\cap I_S$ into subsets $E^t_k$ of points having made $k$ passes through $I\setminus I_S$ before time $t$. After each pass through $I\setminus I_S$, an orbit must spend at least $n_S+1$
iterates within $I_S$ before making its next pass. It results that the index $k$ here is at most $\lfloor \frac{t-2}{n_S+2}\rfloor$. 

In order to estimate the measure of the sets $E_k^t$, we consider separately the cases $k >b \log t$ and $k\leq b \log t$, where $b=\frac{1}{(n_S+1) \gamma \log \sigma^{-1}}$ and $\sigma$ is from Lemma~\ref{lem:exp escape}.

For $k >b \log t$, we observe that every point in $E^t_k$ must spend at least 
$(k-1)(n_S+1)$ iterates in $I_S$ before hitting the hole. Hence we have 
$E^t_k \subset \hI_S^{(k-1)(n_S+1)}$ and using Lemma~\ref{lem:exp escape} 
and the definition of $b$, we get
\[
\sum_{k > b\log t} m(E^t_k) \lesssim \sum_{k > b\log t} \sigma^{(k-1)(n_S+1)}
\lesssim t^{-\frac 1\gamma} ,
\]
as desired.

For $k\leq b \log t$, we first note that the case $k=0$ is easily estimated using 
$E^t_0 \subset \hI_S^t$ and Lemma~\ref{lem:exp escape}. 
From now on, we assume $k\in\{1,\cdots, \lfloor b \log t \rfloor\}$ and observe as before that the subset $E_k^{t,+}\subset E_k^t$ of points whose orbit spends at least $b(n_S+1)\log t$ iterates in $I_S$ up to time $t$ is included in $\hI_S^{b(n_S+1)\log t}$. This inclusion implies $m\left(\bigcup\limits_{k=1}^{\lfloor b\log t \rfloor}E_k^{t,+}\right)\lesssim t^{-\frac 1\gamma}$. 
   
It remains to consider the complementary subset $E_k^{t,-}=E_k^{t}\setminus E_k^{t,+}$ of points whose orbit up to $t$ spends more than $t-b(n_S+1)\log t$ iterates in $I\setminus I_S$. Given $x \in E^{t,-}_k$ and $i\in\{1, \ldots, k\}$, let $n_i > n_S$ be such that $J_{n_i}$ is the element of $\cJ$ where $T^{j}(x)$ begins its $i$th pass through $I\setminus I_S$.  We must have 
\[
n_i - n_S>\frac{t-b(n_S+1)\log t}{k},
\]
for at least one $i\in\{1, \ldots, k\}$, otherwise we would have $\sum\limits_{i=1}^kn_i-n_S\leq t-b(n_S+1)\log t$, which contradicts the definition of $E_k^{t,-}$. Accordingly, we have
\begin{equation}
\label{eq:break down k}
m(E^{t,-}_k)
\leq \sum_{i = 1}^k m\left(x\in E^t_k : n_i - n_S>\frac{t-b(n_S+1)\log t}{k}\right) .
\end{equation}
The sets in this sum can be decomposed using symbolic dynamics. Given two integers $t_1\leq t_2$ and a symbolic word $\theta_{t_1}^{t_2}\in\{L,R\}^{t_2-t_1+1}$, let $J_{\theta_{t_1}^{t_2}}=\bigcap\limits_{\ell=t_1}^{t_2}T^{-\ell}(J_{\theta_\ell})$. We have 
\begin{equation}
\label{eq:ni}
\left\{x\in E^t_k : n_i=n\right\}=\bigcup_{j,\{\theta_\ell\}_{\ell=1}^{j-2},\{\theta_\ell\}_{\ell=j+1}^t} J_{\theta_0^{j-1}}\cap T^{-j}(J_n)\cap J_{\theta_{j+1}^t},
\end{equation}
where, by an abuse of notation\footnote{For this expression to be meaningful, we should decompose the sets $I_S$ and $T^{-(j-1)}(J_S)$ into (standard) cylinder sets prior to define the intermediate cylinder $J_{\theta_1^{j-2}}$. This abuse of notation has no impact on the reasoning here. (The same comment applies to the set $T^{-j}(J_n)$ and to the decomposition in equation \eqref{eq:ni}.) }
\[
J_{\theta_0^{j-1}}:=I_S\cap J_{\theta_1^{j-2}}\cap T^{-(j-1)}(J_S),
\]
and where the union on $\{\theta_\ell\}_{\ell=0}^{j-2}$ (resp.\ $\{\theta_\ell\}_{\ell=j+1}^t$) is taken over all admissible words compatible with $i-1$ (resp.\ $k-i$) passes through $I\setminus I_S$ and avoiding $H$ until at least time $t$. 
The sets in \eqref{eq:ni} are pairwise disjoint; hence it suffices to estimate each quantity $m\left(J_{\theta_0^{j-1}}\cap T^{-j}(J_n)\cap J_{\theta_{j+1}^t}\right)$.

To proceed, notice first that the property $T^{n}(J_n)=J_0$ implies
\[
 J_{\theta_0^{j-1}}\cap T^{-j}(J_n)= J_{\theta_0^{j-1}}\cap T^{-j}(I\setminus I_S)\cap T^{-(j+n)}(J_0).
\]
Moreover, the map $T^{j+n}$ is one-to-one on each element of $T^{-(j+n)}(J_0)$. Assuming $j+n+1\leq t$, and applying the bounded distortion estimate of the proof of Lemma \ref{lem:distortion}, we obtain
\[
\frac{m\left( J_{\theta_0^{j-1}}\cap T^{-j}(J_n)\cap J_{\theta_{j+1}^t}\right)}{m\left(J_{\theta_0^{j-1}}\cap T^{-j}(I\setminus I_S)\cap T^{-(j+n)}(J_0)\right)}\approx
\frac{m\left(T^{j+n}\left(J_{\theta_0^{j-1}}\cap T^{-j}(J_n)\cap J_{\theta_{j+1}^t}\right)\right)}{m\left(T^{j+n}\left(J_{\theta_0^{j-1}}\cap T^{-j}(I\setminus I_S)\cap T^{-(j+n)}(J_0)\right)\right)}.
\]
The second ratio here is equal to 
\[
\frac{m\left(J_0\cap T^{j+n}(J_{\theta_{j+n+1}^t})\right)}{m(J_0)} ,
\]
from where our first estimate follows
\[
m\left(J_{\theta_0^{j-1}}\cap T^{-j}(J_n)\cap J_{\theta_{j+1}^t}\right)\approx m\left(J_0\cap T^{j+n}(J_{\theta_{j+n+1}^t})\right)m\left(J_{\theta_0^{j-1}}\cap T^{-j}(J_n)\right).
\]
Proceeding similarly for the second factor above and using $T(J_S)=I\setminus I_S$, we get 
\[
\frac{m\left(J_{\theta_0^{j-1}}\cap T^{-j}(J_n)\right)}{m\left(J_{\theta_0^{j-1}}\right)}\approx
\frac{m\left(T^{j}\left(J_{\theta_0^{j-1}}\cap T^{-j}(J_n)\right)\right)}{m\left(T^{j}\left(J_{\theta_0^{j-1}}\right)\right)}=\frac{m(J_n)}{a_{n_S}},
\]
from which equation \eqref{eq:spacing} implies 
\begin{equation}
\label{eq:condition}
m\left(J_{\theta_0^{j-1}}\cap T^{-j}(J_n)\cap J_{\theta_{j+1}^t}\right)\approx n^{-\frac{\gamma+1}{\gamma}}m\left(J_0\cap T^{j+n}(J_{\theta_{j+n+1}^t})\right)m\left(J_{\theta_0^{j-1}}\right).
\end{equation}
In the case where $j+n+1> t$ (which happens only when $i=k$), we use the inclusion
\[
J_{\theta_0^{j-1}}\cap T^{-j}(J_n)\cap \bigcup_{\{\theta_\ell\}_{\ell=j+1}^t} J_{\theta_{j+1}^t}\subset J_{\theta_0^{j-1}}\cap T^{-j}(J_n),
\]
to obtain using the relation before equation \eqref{eq:condition}
\[
m\left(J_{\theta_0^{j-1}}\cap T^{-j}(J_n)\cap J_{\theta_{j+1}^t}\right)\lesssim n^{-\frac{\gamma+1}{\gamma}}m\left(J_{\theta_0^{j-1}}\right).
\]
Now use that imposing $i-1$ passes through $I\setminus I_S$ before time $j$ implies at least $(i-1)(n_S+1)$ iterates in $I_S$ before $j$ to obtain the following relation
\[
m\left(\bigcup_{j,\{\theta_\ell\}_{\ell=1}^{j-2}}J_{\theta_0^{j-1}}\right)\subset m\left(\hI_S^{(i-1)(n_S+1)}\right)\lesssim \sigma^{(i-1)(n_S+1)}  .
\]
Similarly, for $j+n+1\leq t$ (otherwise the consideration here is not needed), let $q$ be the number of iterates that the orbits of points in $J_0\cap T^{j+n}(J_{\theta_{j+n+1}^t})$ spend in $I_S$.
Each pass in $I\setminus I_S$ from $i+1$ through $k-1$ must be followed by at least $n_S+1$ iterates in $I_S$; hence $q\geq (k-i-1)(n_S+1)$ (also $q\leq t-(j+n)+1-(k-i-1)$ where the maximum is obtained when each pass in $I\setminus I_S$ consists of a single iterate) and then
by Lemma~\ref{lem:exp escape},
\[
m\left(\bigcup_{j,\{\theta_\ell\}_{\ell=j+1}^t}J_0\cap T^{j+n}(J_{\theta_{j+n+1}^t})\right)\leq\sum_{q\geq (k-i-1)(n_S+1)}m(\hI_S^q)\lesssim\sigma^{(k-i)(n_S+1)}  .
\]
(Notice that this estimate holds even in the case $i=k$.)  Putting these estimates together with \eqref{eq:ni} and \eqref{eq:condition}, we have obtained 
\[
m\left(x\in E^t_k : n_i=n\right)\lesssim \sigma^{(k-1)(n_S+1)}n^{-\frac{\gamma+1}{\gamma}}.
\]
Using the inequality
\[
\sum_{n>n_S+\frac{t-b(n_S+1)\log t}{k}}n^{-\frac{\gamma+1}{\gamma}}\leq \left(n_S+\frac{t-b(n_S+1)\log t}{k}\right)^{-\frac 1\gamma}
\]
it follows from \eqref{eq:break down k} that 
\[
m(E^{t,-}_k)\lesssim \sigma^{(k-1)(n_S+1)}k^\frac{\gamma+1}{\gamma}\left(k n_S+t-b(n_S+1)\log t\right)^{-\frac 1\gamma}.
\]
It remains to sum over $k$. We finally have 
\[
m\left(\bigcup\limits_{k=1}^{\lfloor b\log t \rfloor}E^{t,-}_k\right)\lesssim 
\left(n_S+t-b(n_S+1)\log t\right)^{-\frac 1\gamma}\sum_{k=1}^{\lfloor b\log t \rfloor}
\sigma^{(k-1)(n_S+1)}k^\frac{\gamma+1}{\gamma}\lesssim t^{-\frac 1\gamma},
\]
as desired, where we used $\sum\limits_{k=1}^{b\log t}\sigma^{(k-1)(n_S+1)}k^\frac{\gamma+1}{\gamma}<+\infty$ and $(n_S+t-b(n_S+1)\log t)^{-1}\lesssim t^{-1}$.

\noindent
{\em Case II. Estimate for points in $\hI^t\cap [0, a_{n_S})$.}  
Recall that $[0, a_{n_S})=\bigcup\limits_{n>n_S}J_n$ and by definition of the $J_n$, we have $\hI^t\supset \bigcup\limits_{n>t+h}J_n$ so that  using equation \eqref{eq:spacing} yields 
$m\left(\hI^t\cap \bigcup\limits_{n>t+h}J_n\right)\lesssim (t+h)^{-\frac 1\gamma}\leq t^{-\frac 1\gamma}$.

It remains to estimate $m\left(\hI^t\cap \bigcup\limits_{n=n_S+1}^{t+h}J_n\right)$. For every $n>n_S$, we have 
$T^{n-n_S}(\hI^t\cap J_n)= \hI^{t-n+n_S}\cap J_{n_S}$. Using bounded distortion again, we get
\[
\frac{m\left(\hI^t\cap J_n\right)}{m(J_n)}\approx \frac{m\left(\hI^{t-n+n_S}\cap J_{n_S}\right)}{m(J_{n_S})},
\]
which, together with the inclusion $J_{n_S}\subset I_S$ and the conclusion in Case I, implies
\begin{equation}
\label{eq:inter sum}
m\left(\hI^t\cap \bigcup\limits_{n=n_S+1}^{t+h}J_n\right)\lesssim \sum_{n=n_S+1}^{t+h}n^{-\frac{\gamma+1}{\gamma}}(t-n+n_S)^{-\frac 1\gamma}.
\end{equation}
Lemma~\ref{lem:double} implies that the sum up to $t+h-1$ is bounded above by $\lesssim t^{-\frac1\gamma}$. For the last term, we have $(t+h)^{-\frac{\gamma+1}{\gamma}}(n_S-h)^{-\frac 1\gamma}\lesssim t^{-\frac 1\gamma}$. 

It results that 
$m\left(\hI^t\cap \bigcup\limits_{n=n_S+1}^{t+h}J_n\right)\lesssim t^{-\frac 1\gamma}$ and this concludes the proof of  Lemma~\ref{lem:upper bound} in the case
where $H$ is disjoint from $[\frac 12, \frac 12 + d)$ for some $d>0$.

The case in which $H$ contains $[\frac 12, \frac 12 +d)$ for some $d>0$ is much simpler since
points starting in $[2d,1)$ never enter $[0,2d)$ before escaping.  Thus the estimates of
Case I  with $k=0$ together with Case II imply that the upper bound on $m(\hI^t)$ in this case
is the same.
\end{proof}

With Lemma~\ref{lem:upper bound} established, we are ready to prove 
Theorem~\ref{thm:escape}.

\begin{proof}[Proof of Theorem~\ref{thm:escape}]  Given a density $f$ as in the theorem, let $C_0=\inf\limits_{x\in (0,x_0)}x^\alpha f(x)>0$ and $n_0=\min\{n:a_n\leq x_0\}$. We have $\hI^t\supset \bigcup\limits_{n>t+h}J_n$ for all $t\geq 0$; hence for $t$ sufficiently large so that $t+h\geq n_0$, the equation \eqref{eq:spacing} implies 
\begin{equation}
\label{eq:lower}
\mu_f(\hI^t) = \int_{\hI^t} f \, dm \geq C_0\int_0^{a_{t+h}} x^{-\alpha} \, dx 
\gtrsim (h+t)^{-\frac{1-\alpha}{\gamma}},
\end{equation}
from where the lower bound immediately follows.

For the upper bound, we split $[0,1)$ into 3 intervals: $[0,a_{h+t}) \cup [a_{h+t},a_{n_S})
\cup I_S$ and estimate the intersection of $\hI^t$ with each of these separately.\footnote{For $\alpha=0$, the upper bound also directly follows from the fact that $f$ is uniformly bounded on $I$ together with
Lemma~\ref{lem:upper bound}, namely
\[
\mu_f(\hI^t) \lesssim m(\hI^t) \lesssim t^{-\frac 1\gamma}.
\]
}

On $[0,a_{h+t})$, we estimate,
\[
\int_{\hI^t \cap [0, a_{h+t})} f \, dm \lesssim \int_0^{a_{h+t}} x^{-\alpha} \, dx 
\lesssim (h+t)^{-\frac{1-\alpha}{\gamma}} ,
\]
while on $I_S$, we have using that $f$ is bounded on this set, and Lemma~\ref{lem:upper bound},
\[
\int_{\hI^t \cap I_S} f \, dm \lesssim m(\hI^t) \lesssim  t^{-\frac 1\gamma} .
\]
On $[a_{h+t}, a_{n_S})=\bigcup\limits_{n=n_S+1}^{t+h}J_n$, we proceed as in Case II of the previous proof,
\[
\begin{split}
\int_{\hI^t \cap [a_{h+t}, a_{n_S})} f \, dm = \sum_{n=n_S+1}^{t+h} \int_{\hI^t\cap J_n} f \, dm
& \lesssim \sum_{n=n_S+1}^{t+h} a_n^{-\alpha} \, m\left(\hI^t\cap J_n\right) \\
& \lesssim \sum_{n=n_S+1}^{t+h} n^{-\frac{\gamma+1-\alpha}{\gamma}} (t-n+n_S)^{- \frac1\gamma},
\end{split}
\]
where we have used \eqref{eq:inter sum}.
As before, the last sum (except its last term) is estimated using Lemma~\ref{lem:double}, to give
\[
\sum_{n=n_S+1}^{t+h} n^{-\frac{\gamma+1-\alpha}{\gamma}} (t-n+n_S)^{- \frac1\gamma}\lesssim t^{-\frac{1-\alpha}{\gamma}}.
\]
\end{proof}


\subsection{Properties of the function spaces $\F_p$}
\label{space}
The definition of the quantity $H^p_J$ before Theorem \ref{thm:convergence} implies the following simple facts about the set $\F_p$ ($p\in\R^+$), whose proof we leave to the reader.
\begin{Lem}
\begin{enumerate}
  \item[(1)] $\| \cdot \|_p$ is scale invariant, i.e. $\| C f \|_p = \| f \|_p$ for any $C>0$.
  \item[(2)] For any $J \in \cJ^{(\ell_H)}$, $E$ a subinterval of $J$ and $f \in \F_p$, we have
  \[
  \sup_{x \in E} f(x) \leq e^{H^p_J(f)|E|^p} \inf_{x \in E} f(x) \leq  e^{H^p_J(f)|E|^p} |E|^{-1} \int_E f \,dm .
  \]
  \item[(3)] If $q > p$, then $\F_q \subset \F_p$ (and $\F_q^0 \subset \F_p^0$).
  \end{enumerate}
\label{BASICCLAIM}
\end{Lem}
Now, we equip the set of measures with the topology of weak convergence, consider the ball $\B_p = \{ \mu_f\ :\ f \in \F_p, \| f\|_p \leq 1 \}$ and notice that this ball is not closed. 
Indeed, given $\ell\in\N$, let the density $f_\ell$ be defined by  
\[
f_\ell(x) = \left\{\begin{array}{cl}
a_{\ell}^{-1}&\text{if}\ x \in \bigcup\limits_{n \geq \ell+1} J_n\\
0&\text{elsewhere}. 
\end{array}\right.
\]
Then we have, $f_\ell \geq 0$, $\int f_\ell \, dm = 1$ and $\| f_\ell \|_p = 0$ so that $\mu_{f_\ell}\in \B_p$. However, we clearly have 
\[
\lim_{\ell\to\infty}\mu_{f_\ell}=\delta_0\not\in \B_p.
\]
The Dirac measure at 0 turns out to be the only possible singular component to where sequences in $\B_p$ can accumulate. 
\begin{Lem}
\label{lem:simple facts}
The set $\{(1-s)\mu_f+s \delta_0\ :\ s\in [0,1], f\in \B_p\}$ is compact. 
\end{Lem}
\begin{proof}
Let $\{ \mu_{f_\ell}\}_{\ell\in\N} \subset \B_p$ be an arbitrary sequence.  Since $\int f_\ell \, dm =1$, there exists a subsequence $\{\mu_{f_{\ell_k}}\}$ which converges weakly to a probability measure $\mu_\infty$ on $I$.  Now fix $J \in \cJ^{(\ell_H)}$.  By Lemma~\ref{BASICCLAIM}, the sequence of densities $\{ f_{\ell_k} \}$ is a bounded, equicontinuous family on $J$. By the Arzel\`a-Ascoli theorem, there exists
a subsequence that converges uniformly to a function $f^{(\infty)}_J$ on $J$.\footnote{To be precise,
note that $\| f_\ell \|_p \leq 1$ implies that $f_\ell$ is uniformly continuous on $J$ so that
$f$ can be extended to a continuous function $\bar f_\ell$ on the closure $\bar J$ and 
it still holds that $\| \bar f_\ell \|_p \leq 1$.}
  Note that 
$H^p_J(f^{(\infty)}_J) \leq 1$.

Diagonalizing, we obtain a subsequence $\{ f_{\ell_{k_j}} \}$ converging to 
$f^{(\infty)}_J$ on each $J \in \cJ^{(\ell_H)}$.  Letting $f^{(\infty)}= \sum_J f_J^{(\infty)}$, we have $f^{(\infty)} \geq 0$,
$\| f^{(\infty)} \|_p \leq 1$,  and by Fatou's lemma, $\int f^\infty \, dm \leq 1$.  

Let $s = 1 - \int f^{(\infty)} \, dm$.  If $s<1$, let $f_\infty= (1-s)^{-1} f^{(\infty)}$. By the above observations, we have $\mu_{f_\infty}\in \B_p$.
Since $\{ 0 \}$ is the only accumulation point of the sequence of sets
$\{ J \}_{J \in \cJ^{(\ell_H)}}$, we must have $\mu_\infty = (1-s)\mu_{f_\infty} + s\delta_0$, as required.
\end{proof}
For the next statement, we need to introduce the (nonlinear) normalized transfer operator and its iterates,
\begin{equation}
\hLp_1^t f := \frac{\hLp^t f}{|\hLp^t f|_1},\ \forall t\geq 1.
\label{NORMATRANS}
\end{equation}
Recall that $\F_p^0 = \left\{ f \in \F_p : \exists x_0 \in (0,1) \mbox{ and } \alpha \in [0,1) \mbox{ such that } 
\eqref{eq:order} \mbox{ holds} \right\}$.

\begin{Pro}
\label{pro:invariant}
Let $p\in (0,\frac{\gamma}{\gamma+1}]$. We have $\hLp_1(\F^0_p)\subset \F^0_p$.  In addition, there exist
two constants $C_1, C_2 \geq 0$, such that for every $f \in \F^0_p$,
\[
\| \hLp_1^t f \|_p \leq C_1 \|f\|_p + C_2\ \text{for all}\  t \geq 1.
\]
\end{Pro}

\begin{proof}
Every $J\subset \cJ^{(\ell_H)}$ has at most two pre-images under $\hT$ and each pre-image is included in some element of $\cJ^{(\ell_H)}$. This implies that $\hLp(C^0(\cJ^{(\ell_H)}))\subset C^0(\cJ^{(\ell_H)})$. Also we obviously have 
\[
\hLp f \geq 0\quad\text{and}\quad |\hLp f|_1\leq |f|_1,
\]
for every $f \geq 0$. 

Now, fix $f \in \F^0_p$ and assume $C_0 \ge 1$ and $\alpha \in [0,1)$ are 
such that $C_0^{-1} \le x^\alpha f(x) \le C_0$ for $x \in (0, a_{n_0})$, and $f \le C_0$ on 
$I \setminus [0,a_{n_0})$.  
Without loss of generality, we may assume $n_0 > h + 1$.
Now letting $\hT_L$ and $\hT_R$ denote the left and right branches of $\hT$
respectively, we have
\begin{equation}
\label{eq:Lp bound}
\hLp f(x) = \frac{f(\hT_L^{-1}(x))}{DT(\hT_L^{-1}(x))} + \frac{f(\hT_R^{-1}(x))}{DT(\hT_R^{-1}(x))} .
\end{equation}
If $y = \hT_L^{-1}(x) \in (0, \frac{1}{2})$, then it follows from the definition of $T$ that
$\frac x2 \le y \le x$.  Thus if $x \in (0, a_{n_0-1})$, then by assumption on $f$,
\[
\begin{split}
C_0^{-1}(\hT_L^{-1}(x))^{-\alpha} & \le f(\hT_L^{-1}(x)) \le C_0 (\hT_L^{-1}(x))^{-\alpha} \\
\implies \qquad C_0^{-1} x^{-\alpha} & \le f(\hT_L^{-1}(x)) \le 2C_0 x^{-\alpha} .
\end{split}
\]
Combining this estimate together with \eqref{eq:Lp bound} and using the fact that 
$1 \le DT_L \le 3$ and $DT_R = 2$, we have
\[
\tfrac 13 C_0^{-1} x^{-\alpha} \le \hLp f(x) \le 2C_0 x^{-\alpha} + \tfrac{C_0}{2} \le 3C_0 x^{-\alpha} ,
\]
for all $x \in (0,a_{n_0-1})$, which is the required polynomial bound on the behavior of
$\hLp f$ near 0.  For $x \in [a_{n_0-1}, 1)$, we use the fact that $f$ is bounded
by $C_0$
at both preimages of $x$ so that $\hLp f(x) \le 2C_0$.

Moreover, the lower bound on $\hLp f$ given above implies that
$| \hLp f  |_1 >0$ so that $\hLp_1 f$ is well defined.   
Anticipating the proof below that $\| \hLp f \|_p<+\infty$ for every $f$ with $\|f\|_p<+\infty$ , we obtain  $\hLp_1(\F_p^0)\subset \F_p^0$.

In order to check the estimate on $\| \hLp_1^t f \|_p$, it suffices to
prove the inequality for $\| \hLp^t f \|_p$ due to the scale invariance property
from Lemma~\ref{BASICCLAIM}(1).

Let $f \in \F_p^0$. Fix $n \geq 0$, $J \in \cJ^{(\ell_H)}$, $J \subset J_n$  and $x, y \in J$. Let also $t\geq 1$ and $\{x_i\}$ (resp. $\{y_i\}$) be an enumeration of the pre-images $\hT^{-t}(x)$ (resp.\ $\hT^{-t}(y)$) such that each pair $x_i,y_i$ lies in the same branch of $\hT^{-t}$. Then
\begin{equation}
\label{eq:log}
\log \hLp^t f(x) - \log \hLp^t f(y) = \log \frac{\sum_{i} f(x_i)/ DT^t(x_i)}
{\sum_{i} f(y_i)/ DT^t(y_i)}
\leq \max_{i} \log \frac{f(x_i)}{f(y_i)} + \log \frac{DT^t(y_i)}{DT^t(x_i)},
\end{equation}
where we have used the fact that $\frac{\sum_i b_i}{\sum_i c_i} \leq \max_i \frac{b_i}{c_i}$
for two series of positive terms. The first term on the right hand side of 
\eqref{eq:log} is estimated by
\[
\log \frac{f(x_i)}{f(y_i)} \leq \| f\|_p |x_i-y_i|^p ,
\]
while the second term on the right side of \eqref{eq:log} is estimated by
\[
 \log \frac{DT^t(x_i)}{DT^t(y_i)}  \lesssim |T^t(x_i)-T^t(y_i)|^p= |x-y|^p,
\]
according to Lemma~\ref{lem:distortion}(b).  Putting these two estimates together and using Lemma~\ref{lem:distortion}(a), we obtain that there exist $C_1,C_2>0$ such that
\begin{align*}
\frac{|\log \hLp^t f(x) - \log \hLp^t f(y)|}{|x-y|^p}& \leq \| f \|_p \max_i\frac{|x_i-y_i|^p}{|x-y|^p} + C_2\\
&\leq \|f\|_p C_1 \left(\frac{n}{n+t}\right)^{p\frac{\gamma+1}{\gamma}} + C_2\\
&\leq C_1 \|f\|_p+C_2,
\end{align*}
as desired. 
\end{proof}

The last statement of this section provides examples of unbounded densities that belong to the sets $\F_p^0$.
\begin{Lem}
Let $p\in (0,\frac{\gamma}{\gamma+1}]$. Every function $\frac{f}{|f|_1}$ where $f(x)=x^{-\alpha}$ for all $x\in (0,1)$, $\alpha\in [0,1)$, belongs to $\F_p^0$. The same is true for the density $f_{\mbox{\tiny SRB}}$ associated with the SRB measure.
\label{EXAMPLESSPACE}
\end{Lem}
\begin{proof}
Letting $f(x) = x^{-\alpha}$ for all $x\in (0,1)$, we clearly have $f \in C^0(\cJ^{(\ell_H)})$, $f \geq 0$ and $|f|_1<+\infty$ (and \eqref{eq:order} holds).  Therefore, we only need to check that $\| f\|_p < +\infty$.  

Fix $J \in \cJ^{(\ell_H)}$, $J \subset J_n$ for some $n \geq 0$ and let $x<y \in J$.  Then
\[
 \frac{\log f(x) - \log f(y)}{|x-y|^p} = \alpha \frac{\log(y/x)}{|x-y|^p} \leq \frac{\alpha}{x} |x-y|^{1-p} .
\]
Since $x \in J_n$, we have $\frac{1}{x} \lesssim n^{-\frac 1\gamma}$ and $|x-y| \lesssim n^{-\frac{\gamma+1}{\gamma}}$
by equation \eqref{eq:spacing}.  Thus
\[
H^p_J(f) \lesssim n^{\frac{1-(\gamma+1)(1-p)}{\gamma}}
\]
and the exponent of $n$ is non-positive when $p \leq \gamma/(\gamma+1)$. Taking the supremum over $n$, we have
$\| f\|_p < \infty$ as required.

As for $f_{\mbox{\tiny SRB}}$, observe that the sequence $\{{\mathcal L}^t 1\}_{t\in\N}$ 
(transfer operator for the system without the hole) converges to the density 
$f_{\mbox{\tiny SRB}}$. 
On the other hand, the proof of Proposition \ref{pro:invariant} can be repeated {\sl mutatis mutandis} to conclude ${\mathcal L}(\F_p^0)\subset \F_p^0$ and $\sup_{t\in\N}\| {\mathcal L}^t f \|_p <+\infty$ for all $f\in \F_p^0$. Since $1\in \F_p^0$, it follows from Lemma \ref{lem:simple facts} that $f_{\mbox{\tiny SRB}}\in\F_p$. However, $f_{\mbox{\tiny SRB}}(x)\approx x^{-\gamma}$; hence we must have $f_{\mbox{\tiny SRB}}\in \F_p^0$ as desired. 
\end{proof}


\subsection{Proof of Theorem~\ref{thm:convergence}}
\label{convergence}

The proof relies on the following strengthening of the volume estimate in Lemma~\ref{lem:upper bound}, on the set of points that enters $H$ precisely at time $t$.
\begin{Lem}
\label{lem:precise}
\[
m(\hI^{t-1} \setminus \hI^t) \lesssim t^{-\frac{\gamma+1}{\gamma}} \log t,\ \forall t\geq 2  .
\]
\end{Lem}
The proof is given in Section~\ref{proof of precise} below.
\begin{Rem}
We believe one should be able to eliminate the factor $\log t$ and thus
obtain Lemma~\ref{lem:upper bound} via the identity
$m(\hI^t) = \sum_{i=t}^\infty m(\hI^i \setminus \hI^{i+1})$.  Although we are able to
prove this upper bound in a special case (see Lemma~\ref{lem:one pass}), our current techniques
do not provide this estimate in the general case, so we will use the weaker
version stated above. 
\end{Rem}

Let $f \in \F_p^0$ for some $p>0$. Using Lemma \ref{BASICCLAIM}(3), we may assume without loss of generality that $p \in (0, \frac{\gamma}{\gamma + 1}]$. Assume for now that the 
exponent $\alpha$ of $f$ from \eqref{eq:order} is positive. We are going to  
derive a bound analogous to Lemma~\ref{lem:precise} for $\mu_f$.

Let $n_0$ be the smallest $n$ such that $f(x) \approx x^{-\alpha}$ on $(0,a_{n})$. As in the proof of Theorem~\ref{thm:escape}, we split $[0,1)$ into 3 intervals, namely $[0,a_t)$, $[a_t, a_{n_0})$ and $[a_{n_0},1)$. 

We have $(\hI^{t-1} \setminus \hI^t) \cap [0,a_t) \subset J_{t+h}$.  Thus
\[
\int_{(\hI^{t-1} \setminus \hI^t) \cap [0,a_t)} f \, dm \lesssim \int_{J_{t+h}} x^{-\alpha} \, dx
\lesssim (h+t)^{-\frac{\gamma+1}\gamma+\frac\alpha\gamma} .
\]
Moreover, since $f$ is bounded on $[a_{n_0},1)$, we also have 
\[
\int_{(\hI^{t-1} \setminus \hI^t) \cap [a_{n_0},1)} f \, dm \lesssim m(\hI^{t-1} \setminus \hI^t) 
\lesssim t^{-\frac{\gamma+1}{\gamma}}\log t .
\]
It remains to estimate the $\mu_f$ measure of $\bigcup\limits_{n=n_0+1}^t J_n \cap (\hI^{t-1} \setminus \hI^t)$. 
Since without loss of generality, we may take $n_0 > h$, we have for each $n > n_0$, 
$T^{n-n_0}(J_n \cap (\hI^{t-1} \setminus \hI^t)) = J_{n_0} \cap (\hI^{t-1-n+n_0} \setminus \hI^{t-n+n_0})$.
Using bounded distortion, we obtain,
\[
\frac{m(J_n \cap (\hI^{t-1} \setminus \hI^t))}{m(J_n)} \approx \frac{m(J_{n_0} \cap (\hI^{t-1-n+n_0} \setminus
\hI^{t-n+n_0}))}{m(J_{n_0})} 
\]
 which implies
$m(J_n \cap (\hI^{t-1} \setminus \hI^t)) 
\lesssim n^{-\frac{\gamma+1}{\gamma}} (t-n+n_0)^{-\frac{\gamma+1}{\gamma}}\log(t-n+n_0)$.
Now
\[
\begin{split}
\int_{(\hI^{t-1} \setminus \hI^t) \cap [a_t, a_{n_0})} f \, dm 
& = \sum_{n= n_0+1}^t \int_{(\hI^{t-1} \setminus \hI^t) \cap J_n} f \, dm
\lesssim \sum_{n=n_0+1}^t a_n^{-\alpha} m(J_n \cap (\hI^{t-1} \setminus \hI^t)) \\
& \lesssim \sum_{n=n_0+1}^t 
n^{-\frac{\gamma+1}\gamma+\frac\alpha\gamma} (t-n+n_0)^{-\frac{\gamma+1}{\gamma}}
\log (t-n+n_0) .
\end{split}
\]
Since $\alpha >0$, we may dominate $\log (t-n+n_0)$ by $C (t-n+n_0)^{\alpha/\gamma}$
for some $C>0$.  Now
 using Lemma~\ref{lem:double}, we finally conclude the existence  
of a constant $\bar{C}$ such that 
\begin{equation}
\label{eq:precise alpha}
\mu_f (\hI^{t-1} \setminus \hI^t) \leq \bar{C} t^{-\frac{\gamma+1}\gamma+\frac\alpha\gamma} .
\end{equation}
Notice that using the relation $\mu_f(\hI^t) =\sum\limits_{i=t}^\infty \mu_f(\hI^i \setminus \hI^{i+1})$, 
this inequality implies the estimate $\mu_f(\hI^t)\lesssim t^{-\frac{1-\alpha}{\gamma}}$ for 
$\alpha >0$ without using Lemma~\ref{lem:upper bound}. However, for $\alpha=0$, 
a similar reasoning yields 
$\mu_f(\hI^{t-1} \setminus \hI^t) \lesssim t^{-\frac{\gamma+1}{\gamma}} \log t$ 
from which the conclusion of Theorem \ref{thm:escape} cannot be deduced,
hence the necessity of Lemma~\ref{lem:upper bound}. 

In any case, together with the estimate $\mu_f(\hI^t)\geq C_ft^{-\frac{1-\alpha}{\gamma}}$ from Theorem \ref{thm:escape}, \eqref{eq:precise alpha} yields
\[
1 \geq \frac{\mu_f(\hI^{t+1})}{\mu_f(\hI^t)} = \frac{\mu_f(\hI^t) - \mu_f(\hI^t \setminus \hI^{t+1})}{\mu_f(\hI^t)}
\geq 1 - \frac{\bar{C} t^{-\frac{\gamma+1}\gamma+\frac\alpha\gamma}}{C_f t^{-\frac{1-\alpha}{\gamma}}} 
= 1 - \frac{\bar{C}}{C_f}t^{-1} \to 1 \mbox{ as } t \to +\infty ,
\]
and a similar conclusion holds for $\alpha =0$ by Lemma~\ref{lem:precise}.
Consequently, we have proved the following limit for every $k \geq 1$ (and $\alpha\in [0,1)$),
\begin{equation}
\label{eq:limit 1}
\lim_{t \to +\infty} \frac{\mu_f(\hI^{t+k})}{\mu_f(\hI^t)} 
= \prod_{i=0}^{k-1}\left(\lim_{t \to +\infty} \frac{\mu_f(\hI^{t+i+1})}{\mu_f(\hI^{t+i})}\right)= 1 .
\end{equation}

Now, since $p \in (0, \frac{\gamma}{\gamma + 1}]$, we can apply Proposition~\ref{pro:invariant} to conclude that the sequence $\left\{\frac{\hT^t_\ast\mu_f}{\mu_f(\hI^t)}\right\}_{t\in\N}$ 
is composed of absolutely continuous probability measures with densities in $\F_p^0$. 
By Lemma~\ref{lem:simple facts}, any of its limit points must have the form 
$\mu_\infty = (1-s_\infty)\mu_{f_\infty}+s_\infty\delta_0$ for some $f_\infty\in \F_p$ and
$s_\infty \in [0,1]$. We want to prove that $s_\infty=1$ for any limit point.   

Let $J\in  \mathcal{J}^{(\ell_H)}$, let $g_{t} :=\hLp_1^{t}f$ and consider 
a converging subsequence $\{g_{t_j}\}_{j\in\N}$ with limit point $(1-s_\infty)f_\infty$. 
(Recall that $\hLp_1^t$ is the normalized transfer operator, see equation 
\eqref{NORMATRANS} above.) 
Since $f_\infty\in \F_p$, the convergence 
$g_{t_j} |_J\to (1-s_\infty)f_\infty|_J$ holds in the uniform topology of 
functions defined on this interval. In particular, its integrals against any bounded measurable
function converge as well on each $J \in \mathcal{J}^{(\ell_H)}$.  

Fixing $k \geq 1$, note that the set $\bigcup\limits_{i=0}^k T^{-i}(H)$ is bounded away from 0 and thus
intersects only finitely many elements of $\mathcal{J}^{(\ell_H)}$.  Thus
the sequence $\{g_{t_j}\}_{j\in\N}$ converges uniformly on this set as well.
Now, we have 
\[
\frac{\mu_f(\hI^{k +t_j})}{\mu_f(\hI^{t_j})} = \frac{|\hLp^{t_j+k} f|_1}{|\hLp^{t_j} f|_1} = \int_I \hLp^k g_{t_j} \, dm
= \int_{\hI^k} g_{t_j} \, dm
= 1 - \int_{\cup_{i=0}^k T^{-i}(H)} g_{t_j} \,  dm ,
\]
using the fact that $\int_{I} g_{t_j} \, dm =1$.

Since the limit of the above expression is 1 by \eqref{eq:limit 1} and the convergence of 
$g_{t_j}$ to $(1-s_\infty)f_\infty$ is uniform on each $J$, we must have $f_\infty \equiv 0$ on 
$\bigcup\limits_{i=0}^k T^{-i}(H)$.  Since $f_\infty \in \F_p$ is log-H\"older continuous
on each $J \in \mathcal{J}^{(\ell_H)}$, we conclude that
$f_\infty \equiv 0$ on any $J$ such that $J \cap \left(\bigcup\limits_{i=0}^k T^{-i}(H)\right) \neq \emptyset$.
Since this holds for all $k$, the transitivity of $T$ implies that we must have 
$f_\infty \equiv 0$ on all $J \in \mathcal{J}^{(\ell_H)}$, {\sl viz.}\ $s_\infty=1$. Since the subsequence is arbitrary, it follows that $s_\infty=1$ for any limit point as desired. 
Theorem \ref{thm:convergence} is proved.


\subsection{Proof of Lemma~\ref{lem:precise}}
\label{proof of precise}
We first prove the following auxiliary result. Given $t\geq 1$ and $0\leq s<t$, let 
\[
E^t=\left\{x\in I\ :\ \min\{k\geq 0\ :\ T^k(x)\in J_h\}=t\right\},
\]
and
\[ 
E_s^{s+t}=\left\{x\in T^{-s}(J_h)\ :\ \min\{k> 0\ :\ T^{s+k}(x)\in J_h\}=t\right\}.
\]
\begin{Lem}
\label{lem:one pass}
$m(E^t)\lesssim t^{-\frac{\gamma+1}{\gamma}}$ and $m(E_s^{s,t})\lesssim t^{-\frac{\gamma+1}{\gamma}}$.
\end{Lem}
\begin{proof}
The set $E^t$ consists of a collection of intervals in $T^{-t}(I)$.  Using symbolic dynamics as in equation \eqref{eq:ni} above, we label these intervals according to their itinerary up to the first time $k$ ($0 \leq k \leq t$) when they enter $[0, a_{h-1})$, {\sl viz.}\
\begin{equation}
E^t=\bigcup_{k=0}^t\bigcup_{\{\theta_\ell\}_{\ell=0}^{k-1}} J_{\theta_0^{k-1}}\cap T^{-k}(J_{h+t-k}),
\label{eq:landing}
\end{equation}
where the words $\{\theta_\ell\}_{\ell=0}^{k-1}$ are such that $J_{\theta_\ell^{k-1}}\subset [a_{h-1},1)$ for all $0\leq \ell\leq k-1$. Notice that the term $k=0$ actually reduces to $J_{h+t}$. 

Using bounded distortion for the map $T^k$, we get 
\[
\frac{m\left(J_{\theta_0^{k-1}}\cap T^{-k}(J_{h+t-k})\right)}{m\left(J_{\theta_0^{k-1}}\cap T^{-k}([0,a_{h-1}))\right)}\approx\frac{m\left(J_{h+t-k}\right)}{m\left([0,a_{h-1})\right)}\approx (h+t-k)^{-\frac{\gamma+1}{\gamma}}.
\]
Moreover, recall the induced map $S$ and the associated sets $\hI_S^{t}$ defined at the beginning of section \ref{escape}, which we now consider for $n_S=h-1$. For $k\geq 1$,  we actually have
\[
\bigcup_{\{\theta_\ell\}_{\ell=0}^{k-1}} J_{\theta_0^{k-1}}\cap T^{-k}([0,a_{h-1}))\subset \hI_S^{k-1},
\]
hence Lemma~\ref{lem:exp escape} implies the existence of $\sigma_h\in (0,1)$ such that 
\[
m\left(\bigcup_{\{\theta_\ell\}_{\ell=0}^{k-1}} J_{\theta_0^{k-1}}\cap T^{-k}([0,a_{h-1}))\right)\lesssim \sigma_h^k.
\]
For $k=0$, we obviously have $m([0,a_{h-1}))\leq 1$. Grouping all terms together, we finally get 
\[
m(E^t)\lesssim\sum_{k=0}^t (h+t-k)^{-\frac{\gamma+1}{\gamma}} \sigma_h^k.
\] 

We now show that this last sum is $\lesssim t^{-\beta}$ where $\beta=\frac{\gamma+1}{\gamma}$. The terms $k=0$, $k=1$ and $k=t$ satisfy this estimate. Moreover, the remaining sum from $k=2$ to $t-1$ is not larger than the following integral which we compute by integration by parts, 
\begin{align*}
\int_1^t (h+t-x)^{-\beta} \sigma_h^{x} \, dx \;
= \; \; &\frac{\sigma_h(h+t-1)^{-\beta}}{\log\sigma_h^{-1}}-\frac{h^{-\beta}\sigma_h^t}{\log\sigma_h^{-1}}\\
&+\int_1^{\frac{t}{2}} \frac{(h+t-x)^{-\beta-1} \sigma_h^{x}}{\log\sigma_h^{-1}} \, dx+\int_{\frac{t}{2}}^t \frac{(h+t-x)^{-\beta-1} \sigma_h^{x}}{\log\sigma_h^{-1}} \, dx
\end{align*}
The first two terms in the right hand side here are $\lesssim t^{-\beta}$. Moreover, on the first integral, the integrand is at most $\frac{\sigma_h}{\log\sigma_h^{-1}}(h+\frac{t}2)^{-\beta-1}$; hence the desired estimate follows. On the second integral, the integrand is at most $\frac{h^{-\beta-1}}{\log\sigma_h^{-1}}\sigma_h^t$ and then the integral is also dominated as $\lesssim t^{-\beta}$. The proof of the estimate on $m(E^t)$ is complete. 

For $m(E_s^{s+t})$, the argument is similar. We decompose the set as follows
\[
E_s^{s+t}=\bigcup_{k=h+1}^t\bigcup_{\{\theta_\ell\}_{\ell=0}^{s+k-1}} J_{\theta_0^{s+k-1}}\cap T^{-(k+s)}(J_{h+t-k}),
\]
where now the words $\{\theta_\ell\}_{\ell=0}^{s+k-1}$ are such that $J_{\theta_s^{s+k-1}}\subset J_h$ and $J_{\theta_\ell^{s+k-1}}\subset [a_{h-1},1)$ for all $s< \ell\leq s+k-1$. Using the same bounded distortion estimates as above, we get 
\[
m(E_s^{s+t})\lesssim\sum_{k=h+1}^t (h+t-k)^{-\frac{\gamma+1}{\gamma}} \sigma_h^k,
\] 
from which the conclusion immediately follows. 
\end{proof}

We will prove Lemma~\ref{lem:precise} by keeping track of passes
through $J_h$ and concatenating estimates of the
form in Lemma~\ref{lem:one pass}.  Unfortunately, in estimating the
contribution of points that make $k$ passes through $J_h$ before entering $H$,
we obtain a factor $C_0^k$, which is potentially disastrous since $k$ can
increase with $t$.  In order to overcome the effect of this constant, we 
slightly weaken the rate of decay as in the statement of the lemma.

\begin{proof}[Proof of Lemma~\ref{lem:precise}.]
Let $C_0>0$ be the constant in the estimates of Lemma~\ref{lem:one pass} and choose $n_0$ large enough that
\begin{equation}
\label{eq:n_0}
\rho := C_0  2^{\frac{1}{\gamma}+2} \gamma (n_0-1)^{-\frac{1}{\gamma}} < 1 .
\end{equation}

We call a return to $J_h$ {\em long} if there have been at least $n_0$ iterates since the last entry
to $J_h$ (or for the first entry, if it occurs after time $n_0-1$). A return that is not long is called {\em short}.

Given $x \in \hI^{t-1} \setminus \hI^t$, consider the decomposition of $[0,t]$ into segments, labelled
alternating $s_1, \ell_1, s_2, \ell_2, \ldots , s_k, \ell_k, s_{k+1}$ where the segments $\ell_i$ are long returns to $J_h$ and the segments $s_i$ are comprised of one or more short returns to $J_h$ (but the total length $s_i$
of the segment may be greater than $n_0$). Note we must have $\ell_i \geq n_0$ for each $i$ 
but some of the $s_i$ may be 0. 

Let $(s_1,\ldots, s_{k+1}; \ell_1,\ldots, \ell_k; t)\subset \hI^{t-1} \setminus \hI^t$ denote the
set of points with the specified trajectory which fall into the hole 
at time $t$ with $k$ long returns to $J_h$. Let also $\ell = \sum\limits_{i=1}^k \ell_i$.

Recall once again the induced map $S$ and the associated sets $\hI_S^{t}$, which we now consider for $n_S=n_0 + h$. Lemma~\ref{lem:exp escape} implies the existence of $C_3>0$ and $\sigma_0<1$ such that $m(\hI_S^t) \leq C_3 \sigma_0^t$ for all $t$. 
Now choose $C_4 > 0$ sufficiently large such that
\[
C_4 \log \sigma_0^{-1} \geq \frac{\gamma+1}{\gamma} .
\]
We divide our estimate into two cases.

\medskip
\noindent
{\em Case 1:} $t-\ell >  C_4 \log t $.  During short returns, iterates remain in the interval $[a_{n_0+h},1)$. Therefore, the orbits of every point in $(s_1,\ldots, s_{k+1}; \ell_1,\ldots, \ell_k; t)$ spend at least $t - \ell$ iterates inside $[a_{n_0+h},1)$, {\sl i.e.}\ we have $(s_1,\ldots, s_{k+1}; \ell_1,\ldots, \ell_k; t)\subset \hI_S^{t-\ell}$ which implies
\[
m(x \in \hI^{t-1} \setminus \hI^t : t - \ell >  C_4 \log t) \leq C_3 \sigma_0^{C_4 \log t}
\leq C_3 t^{-\frac{\gamma+1}{\gamma}} .
\]

\medskip
\noindent
{\em Case 2:} $t - \ell \leq  C_4 \log t $.  
We fix $k \geq 1$ and estimate the contribution to $\hI^{t-1} \setminus \hI^t$
from points making $k$ long returns to $J_h$.  Now
\[
\begin{split}
m(s_1, \ldots, s_{k+1}; \ell_1, \ldots, \ell_k; t)
= & \; m(  s_{k+1} \mid (s_1, \ldots, s_k; \ell_1, \ldots ,\ell_k)) \\
& \cdot \prod_{i=1}^k m(\ell_i \mid (s_1, \ldots, s_i; \ell_1, \ldots ,\ell_{i-1}))
m(s_i \mid (s_1, \ldots, s_{i-1}; \ell_1, \ldots ,\ell_{i-1})),
\end{split}
\]
where the truncated lists $(s_1, \ldots ,s_i; \ell_1, \ldots ,\ell_{i-1})$ denote the set of points
whose entries to $J_h$ follow the specified itinerary.  Note that by Lemma~\ref{lem:one pass} 
 we have 
\[
m(\ell_i \mid (s_1, \ldots ,s_i; \ell_1, \ldots ,\ell_{i-1})) \leq C_0 \ell_i^{-\frac{\gamma+1}{\gamma}} .
\]
Also, we have 
\[
\sum_{\mbox{\scriptsize relevant $s_i$}}m(s_i \mid (s_1, \ldots, s_{i-1}; \ell_1, \ldots, \ell_{i-1})) 
\leq 1
\]
for each $i$.  Since each sum over $s_i$ holds for $s_1, \ldots, s_{i-1}$ fixed, we apply
the bound above recursively to obtain,
\[
\sum_{s_1,\ldots,s_{k+1}}\prod_{i=1}^{k+1}m(s_i \mid (s_1, \ldots, s_{i-1}; \ell_1, \ldots, \ell_{i-1})) \leq 1,
\]
where the sum runs over relevant $s_1,\ldots,s_{k+1}$. 
Finally, summing the products over the relevant lengths $\ell_i$, we must 
estimate the following expression, where as before $\beta = \frac{\gamma+1}{\gamma}$,
\begin{equation}
\label{eq:big sum}
\sum_{\ell =\lceil t - C_4 \log t\rceil}^t \sum_{\ell_1 = n_0}^{\ell - (k-1)n_0}
\sum_{\ell_2= n_0}^{\ell - \ell_1 - (k-2)n_0} \ldots
\sum_{\ell_{k-1}=n_0}^{\ell - \ell_1 - \ldots - \ell_{k-2} - n_0}
C_0^k (\ell_1 \cdots \ell_{k-1})^{-\beta} 
(\ell - \ell_1 - \ldots - \ell_{k-1})^{-\beta},
\end{equation}
where we have used the fact that $\ell_k = \ell - \ell_1 - \ldots - \ell_{k-1}$.
 We will estimate the iterated sums
one at a time and show the calculation in detail in order to verify that we 
can control the effect of the constant $C_0^k$.

Let $s = \ell - \ell_1 - \ldots - \ell_{k-2}$.  Then the first sum we must estimate is
\begin{equation}
\label{eq:one sum}
\begin{split}
\sum_{\ell_{k-1}=n_0}^{s-n_0} & (s - \ell_{k-1})^{-\beta} 
\ell_{k-1}^{-\beta}
\leq 2\sum_{\ell_{k-1} = n_0}^{\lfloor s/2\rfloor} (s- \ell_{k-1})^{-\beta} 
\ell_{k-1}^{-\beta} 
\leq 2^{1+\beta}s^{-\beta}\sum_{\ell_{k-1} = n_0}^{\lfloor s/2\rfloor} \ell_{k-1}^{-\beta}    \\
& \! \! \! \leq 2^{1+\beta}s^{-\beta}\int_{n_0-1}^{\infty} x^{-\beta} \, dx
\leq 2^{1+\beta} s^{-\beta} \tfrac{(n_0-1)^{1-\beta}}{\beta-1}
= 2^{\frac{1}{\gamma}+2}s^{-\beta}
\gamma (n_0-1)^{-\frac{1}{\gamma}} \leq C_0^{-1} \rho s^{-\beta}
\end{split}
\end{equation}
where the last inequality follows from equation \eqref{eq:n_0}.  Turning now to the second sum in \eqref{eq:big sum}, we must
estimate
\[
\sum_{\ell_{k-2} = n_0}^{\ell - \ell_1 - \ldots - \ell_{k-3} - 2n_0}
(\ell - \ell_1 - \ldots - \ell_{k-2})^{-\beta} \ell_{k-2}^{-\beta},
\]
and setting $s = \ell - \ell_1 - \ldots - \ell_{k-3}$, we see that this sum has the same
form as that in \eqref{eq:one sum} and thus satisfies the same bound (note, we are 
not hurt by the fact that we are subtracting $2n_0$ rather than $n_0$ in the upper limit of this sum
since subtracting $2n_0$ just makes the sum smaller.  Proceeding this way $k-3$ more times,
we arrive at the expression,
\begin{align*}
m(x \in \hI^{t-1} \setminus \hI^t : k \mbox{ long passes with $t - \ell  \leq  C_4 \log t $}) 
&\leq 
\sum_{\ell = \lceil t - C_4\log t\rceil }^t C_0 \rho^{k-1} \ell^{-\beta}\\
&\leq C_0 \rho^{k-1} \lfloor C_4 \log t+1\rfloor \lceil t - C_4 \log t\rceil^{-\frac{\gamma+1}{\gamma}} 
\end{align*}
Finally, summing over $k$ yields the existence of $C>0$ such that 
\[
m(\hI^{t-1} \setminus \hI^t : t - \ell \leq C_4 \log t) \leq C t^{-\frac{\gamma+1}{\gamma}} \log t .
\]
Putting together our estimates from Cases 1 and 2 completes the proof of the lemma.
\end{proof}


\subsection{General open systems - proof of Theorem~\ref{thm:mean conv}}
\label{cesaro}
The proof requires a preliminary statement. We say a set 
$A \subset \N$ has density $\rho\in [0,1]$ if $\lim\limits_{n \to +\infty}\frac{\#(A \cap [1,n])}{n} = \rho$, where
\# denotes the cardinality of a set. 

Define $\beta_t = \frac{\mu(\hX^{t+1})}{\mu(\hX^{t})}$ for $t\geq 0$ and $\beta_{-1}=\mu(\hX^{0})$.  We obviously have $\beta_t \leq 1$ for all $t\geq -1$ and the desired preliminary statement specifies the density of any set of indexes $t$ in which $\beta_t$ remains bounded 
away from 1. 
\begin{Lem}
\label{lem:rate}
For $\lambda \in (0,1)$, define $A_\lambda = \{ k \in \N : \beta_k \leq \lambda \}$.  Then
$A_\lambda$ must have zero density.
\end{Lem}

\begin{proof}
Fix $\lambda \in (0,1)$ and note that if
$A_\lambda$ is finite, it has zero density.  So assume $A_\lambda$ is infinite.
Then $\limsup\limits_{\substack{k \to +\infty \\ k \in A_\lambda}} \beta_k \leq \lambda$.
Equation \eqref{eq:slow} implies 
\begin{equation}
\limsup_{t\to+\infty}\mu(\hX^t)^\frac{1}{t}=\limsup_{t\to+\infty}\exp \left(\frac 1t \log \mu(\hX^t) \right) =1.
\label{LIMITBETA}
\end{equation}
To derive a  contradiction, assume $A_\lambda$ has density $\rho>0$. 
The relation $\mu(\hX^t)=\prod\limits_{k=-1}^{t-1} \beta_k$ then implies 
\[
\mu(\hX^t)^\frac{1}{t}\leq \lambda^\frac{\#(A_\lambda \cap [0,t-1])}{t}\ \text{for all}\ t>0,
\]
from which it follows that $\limsup\limits_{t\to\infty}\mu(\hX^t)^\frac{1}{t}\leq \lambda^\rho<1$, contradicting equation \eqref{LIMITBETA}.
\end{proof}

Passing now to the proof of the theorem, we consider the sequence of probability measures $\{ \mu_t \}_{t \in \N}$ where $\mu_t$ is defined by equation \eqref{CESARO}. By assumption on $X$, this sequence is compact and any limit point $\mu_\infty$ is necessarily a (regular Borel)
probability measure.

Our proof has two steps: (A) $\mu_\infty$ gives zero weight to 
$int \left( \bigcup\limits_{j=0}^i T^{-j}(H) \right)$ for every $i \geq 0$;
(B) if $\mu_\infty$  gives zero weight to the discontinuity set of $\hT$, then
$\mu_\infty$ is an invariant measure for $T$.

Since $T$ is assumed to be nonsingular with respect to $\mu$, the associated transfer operator $\Lp$ acting on $L^1(\mu)$ is well-defined. As before, we also consider the transfer operator $\hLp$ for the system with hole defined by $\hLp^t f = \Lp^t (f 1_{\hX^t})$ for all $t\geq 1$. 

\noindent
(A)  First note that $\hT^t_\ast\mu$ is supported on $X \setminus H$ for each $t \geq0$ so that
$\mu_\infty(H) =0$.  Now define $g_k = \frac{\hLp^k 1}{\mu(\hX^k)}$ for $k \geq 1$.  
For each $i \geq 1$, $k \geq 1$, we have (using the composition property pointed out at the
end of Section~\ref{setting})
\[
\beta_{k+i-1} \cdots \beta_{k}= 
\frac{\int_X \hLp^i(\hLp^k 1) \, d\mu}{\mu(\hX^k)}
= \int_{\hX^i}  g_k \, d\mu = 1 - \int_{\bigcup\limits_{j=0}^{i} T^{-j}(H)} g_k \, d\mu ,
\]
since $\int_{X \setminus H} g_k =1$.  Choose $\lambda \in (0,1)$ and let $A_\lambda$ 
be as defined in Lemma~\ref{lem:rate}.  Given $i\geq 1$, let
\begin{equation}
\label{eq:A}
A_\lambda^i=A-[0,i-1]
\end{equation}
denote the translates of elements of $A_\lambda$ by
some integer at most $i-1$.  Note that $A_\lambda^i$ still has frequency zero.
Then since $\liminf\limits_{k\to +\infty} \beta_{k+i-1} \cdots \beta_k \geq \lambda^i$ 
as long as we avoid the 
exceptional set $A_\lambda^i$, we have
\[
\limsup_{\substack{k \to +\infty \\ k \notin A_\lambda^i}} \int_{\bigcup\limits_{j=0}^i T^{-j}(H)} g_k \, d\mu
\leq 1 - \lambda^i \, .
\]
Since $d\mu_{t_j} =\frac{1}{t_j} \sum\limits_{k=0}^{t_j-1} g_k d\mu$ and $A_\lambda^i$ has
zero density, we conclude 
$\mu_\infty\Big(int \big( \bigcup\limits_{j=0}^{i} T^{-j}(H) \big) \Big) \leq 1 - \lambda^i$.  
Since $\lambda$ is arbitrary, it must be that 
$\mu_\infty\Big(int \big( \bigcup\limits_{j=0}^{i} T^{-j}(H) \big) \Big)=0$.
Since $i$ is arbitrary, $\mu_\infty$ must be supported on
$X \setminus \Big( \bigcup\limits_{i=0}^\infty int \big( \bigcup\limits_{j=0}^{i} T^{-j}(H) \big) \Big)$, 
and therefore is singular with respect to $\mu$ by assumption \eqref{eq:slow}.

\noindent
(B)  For any test function $\vf \in C^0(X)$, we have
\[
\begin{split}
\hT_* \mu_{t_j}(\vf) & = \frac{1}{t_j} \sum_{k=0}^{t_j-1} \frac{\int \hLp^{k+1}1 \, \vf \, d\mu}{\mu(\hX^k)}
=  \frac{1}{t_j} \sum_{k=0}^{t_j-1} \frac{\int \hLp^{k+1}1\, \vf \, d\mu}{\mu(\hX^{k+1})} \beta_{k} \\
& = \frac{1}{t_j} \sum_{k=0}^{t_j-1} \frac{\int \hLp^k 1\, \vf \, d\mu}{\mu(\hX^{k})} \beta_{k-1}
+ \frac{1}{t_j} \left( \frac{\int \hLp^{t_j} 1 \, \vf \, d\mu}{\mu(\hX^{t_j-1})} - \int_{\hX^0} \vf \, d\mu\right).
\end{split}
\]
Again choose $\lambda \in (0,1)$, define $A_\lambda$ as in Lemma~\ref{lem:rate}
and $A^1_\lambda$ as in \eqref{eq:A}.
The discontinuity set of $\vf \circ \hT$ is contained in the discontinuity set of $\hT$, and since
we assume that $\mu_\infty$ gives zero measure to this set, we may pass to the limit 
(see e.g.~\cite[Theorem 5.2(iii)]{billingsley}) to obtain
\[
\begin{split}
|\hT_* \mu_\infty(\vf) - \mu_\infty(\vf)| 
& \leq \lim_{j \to +\infty} \frac{1}{t_j} \sum_{\substack{k=0 \\ k \in A^1_\lambda}}^{t_j-1}
\frac{|\int \hLp^k1 \, \vf \, d\mu|}{|\hLp^k1|_1} |\beta_{k-1}-1| 
+ \frac{1}{t_j} \left( \frac{|\vf|_\infty |\hLp^{t_j}1|_1}{|\hLp^{t_j-1}1|_1} + |\vf|_\infty \right)  \\
& \qquad + \lim_{j \to +\infty} \frac{1}{t_j} \sum_{\substack{k=0 \\ k \notin A^1_\lambda}}^{t_j-1}
\frac{|\int \hLp^k1 \, \vf \, d\mu|}{|\hLp^k1|_1} |\beta_{k-1}-1| \\
& \leq |\vf|_\infty (1-\lambda)
\end{split}
\]
where we have used the fact that the set of $k$ for which $\beta_{k-1} \leq \lambda$
has density 0, $\beta_k \leq 1$,
and $|\int \hLp^k 1\, \vf \, d\mu|/|\hLp^k1|_1 \leq |\vf|_\infty$ for each $k$. 
Since $\lambda$ is arbitrary, we must have $\hT_* \mu_\infty(\vf) = \mu_\infty(\vf)$.

Furthermore, since $\partial H \cup \hT^{-1} (\partial H)$ 
is contained in the discontinuity set for $\hT$, 
by assumption
we have $\mu_\infty(\partial H \cup \hT^{-1} (\partial H))=0$ and 
this together with part (A) implies $\mu_\infty(H \cup T^{-1} (H)) =0$, i.e.~$\mu_\infty(\hX^1) = 1$.
Thus for $\vf \in C^0(X)$,
\[
\mu_\infty(\vf) = \hT_* \mu_\infty(\vf) = \mu_\infty(\vf \circ T \cdot 1_{\hX^1})
= \mu_\infty(\vf \circ T) = T_*\mu_\infty(\vf) .
\]
Since this identity holds for all $\vf \in C^0(X)$ and $\mu_\infty$ is a regular Borel measure, 
it follows from the uniqueness statement of the 
Riesz-Markov Theorem \cite[Theorem 13.23]{royden} that $T_*\mu_\infty = \mu_\infty$
so that $\mu_\infty$ is an invariant measure for $T$.  Now the
 invariance of $\mu_\infty$ implies
$\mu_\infty(T^{-j}(H)) = 0$ for each $j \in \mathbb{N}$ so that $\mu_\infty(\hX^\infty)=1$.
\bigskip

\noindent
{\bf Acknowledgements} 

\noindent
We are grateful to L.-S.~Young for stimulating discussions. B.F.~thanks the Courant Institute for hospitality, where part of this work was done. M.D.~is partially supported by NSF grant DMS 1101572 and B.F.~is partially supported by EU FET Project No. TOPDRIM 318121.


\end{document}